\documentclass[11pt,reqno]{amsart}
\usepackage{amsmath, amsthm, amscd, amsfonts, amssymb, graphicx, color}
\usepackage[bookmarksnumbered, colorlinks, plainpages]{hyperref}

\newtheorem{theorem}{Theorem}[section]

\theoremstyle{definition}
\newtheorem{definition}[theorem]{Definition}

\theoremstyle{remark}
\newtheorem{remark}[theorem]{Remark}

\numberwithin{equation}{section}

\newcommand{\be}{\begin{equation}}
\newcommand{\ee}{\end{equation}}
\newcommand{\bea}{\begin{eqnarray}}
\newcommand{\eea}{\end{eqnarray}}
\begin{document}

\begin{center}
{\large{\textbf{Almost Kenmotsu metric as quasi Yamabe soliton}}}

\end{center}
\vspace{0.1 cm}

\begin{center}

Dibakar Dey and Pradip Majhi\\
Department of Pure Mathematics,\\
University of Calcutta,
35 Ballygunge Circular Road,\\
Kolkata - 700019, West Bengal, India,\\
E-mail: deydibakar3@gmail.com, mpradipmajhi@gmail.com\\
\end{center}

\vspace{0.5 cm}
\textbf{Abstract:} In the present paper, we characterize  a class of almost Kenmotsu manifolds admitting quasi Yamabe soliton. It is shown that if a $(k,\mu)'$-almost Kenmotsu manifold admits a quasi Yamabe soliton $(g,V,\lambda,\alpha)$ with $V$ pointwise collinear with $\xi$, then (1) $V$ is a constant multiple of $\xi$, (2) $V$ is a strict infinitesimal contact transformation and (3) $(\pounds_
V h')X = 0$ for any vector field $X$. Finally an illustrative example is presented to support the result.\\



\section{Introduction}
In 1989, Hamilton \cite{ham1} proposed the concept of Yamabe flow defined as
\bea
\nonumber \frac{\partial g}{\partial t} = - rg,
\eea
where $r$ is the scalar curvature of the manifold. A Riemannian manifold $(M,g)$ is said to admit a Yamabe soliton $(g,V,\lambda)$ if
\bea
\frac{1}{2}\pounds_V g = (r - \lambda)g, \label{1.1}
\eea
where $\pounds_V$ denotes the Lie derivative along $V$ and $\lambda$ is a constant. A Yamabe soliton is said to be shrinking, steady or expanding according as $\lambda < 0$, $\lambda = 0$ or $\lambda > 0$ respectively. If $V$ is a gradient of some smooth function $f : M \rightarrow \mathbb{R}$, that is, $V = Df$, where $D$ is the gradient operator, then the Yamabe soliton is said to be a gradient Yamabe soliton and then (\ref{1.1}) reduces to
\bea
\nabla^2 f = (r - \lambda)g, \label{1.2}
\eea
where $\nabla^2 f$ is the Hessian of $f$.\\

\noindent Extending the notion of Yamabe soliton, Chen and Deshmukh \cite{cd1} introduced the notion of quasi Yamabe soliton which can be defined on a Riemannian manifold as follows:
\bea
\frac{1}{2}(\pounds_V g)(X,Y) = (r - \lambda)g(X,Y) + \alpha V^\#(X)V^\#(Y), \label{1.3}
\eea
where $V^\#$ is the dual 1-form of $V$, $\lambda$ is a constant and $\alpha$ is a smooth function. If $V$ is a gradient of some smooth function $f$, then the above notion is called quasi Yamabe gradient soliton and then (\ref{1.3}) can be written as
\bea
\nabla^2 f = (r - \lambda)g + \alpha df \otimes df. \label{1.4}
\eea
The quasi Yamabe soliton is said to be shrinking, steady or expanding according as $\lambda < 0$, $\lambda = 0$ or $\lambda > 0$ respectively. In \cite{Li}, it is proved that the scalar curvature of a compact quasi Yamabe soliton is constant. In \cite{lf}, the author proved that a quasi Yamabe gradient soliton on a complete non-compact manifold has warped product structure. In \cite{ag}, Ghosh proved a similar result for a Kenmotsu manifold.
Since a $(k,-2)'$-almost Kenmotsu manifold is locally isometric to a warped product space and the scalar curvature is also constant ( see Theorem 4.2 of \cite{dp}), the results proved in \cite{Li} and \cite{ag} are trivially holds in a $(k,-2)'$-almost Kenmotsu manifolds admitting quasi Yamabe gradient soliton. This motivates us to consider only quasi Yamabe soliton on $(k,\mu)'$-almost Kenmotsu manifolds. \\

\noindent The paper is organized as follows:  In section 2, some preliminary results on almost Kenmotsu manifolds are presented.  Section 3 deals with $(k,\mu)'$-almost Kenmotsu manifolds admitting quasi Yamabe soliton. Section 4 is contains an example to support the main result.

\section{\textbf{Preliminaries}}

A $(2n+1)$-dimensional almost contact metric manifold is a differentiable manifold $M$     together with a structure $(\phi,\xi,\eta,g)$ satisfying
\be
\phi^{2}=-I+\eta\otimes\xi,\;\; \eta(\xi)=1, \;\; \phi \xi = 0, \;\; \eta \circ \phi = 0,\label{2.1}
\ee
\be
 g(\phi X,\phi Y)=g(X,Y)-\eta(X)\eta(Y), \label{2.2}
\ee
for any vector fields $X$, $Y$ on $M$, where $\phi$ is a $(1,1)$-tensor field, $\xi$ is a unit vector field, $\eta$ is a 1-form, $g$ is the Riemannian metric and $I$ is the identity endomorphism. The fundamental 2-form $\Phi$ on an almost contact metric manifold is defined by $\Phi(X,Y)=g(X,\phi Y)$ for any $X$, $Y$ on $M$. Almost contact metric manifold such that $\eta$ is closed and $d\Phi=2\eta\wedge\Phi$ are called almost Kenmotsu manifolds (see \cite{dp}, \cite{pv}). \\

\noindent Let $M$ be a $(2n + 1)$-dimensional almost Kenmotsu manifold. We denote by $h=\frac{1}{2}\pounds_{\xi}\phi$ and $l=R(\cdot, \xi)\xi$ on $M$. The tensor fields $l$ and $h$ are symmetric operators and satisfy the following relations \cite{pv}:
 \be
h\xi=0,\;l\xi=0,\;tr(h)=0,\;tr(h\phi)=0,\;h\phi+\phi h=0, \label{2.3}
 \ee
\be
 \nabla_{X}\xi=X - \eta(X)\xi - \phi hX(\Rightarrow \nabla_{\xi}\xi=0), \label{2.4}
\ee
 \be
  \phi l \phi-l = 2(h^{2} - \phi^{2}),\label{2.5}
 \ee
 \be
   R(X,Y)\xi = \eta(X)(Y - \phi hY) - \eta(Y)(X - \phi hX)+(\nabla_{Y}\phi h)X - (\nabla_{X}\phi h)Y, \label{2.6}
 \ee
 for any vector fields $X,\;Y$ on $M$. The $(1,1)$-type symmetric tensor field $h'=h\circ\phi$ is anti-commuting with $\phi$ and $h'\xi=0$. Also it is clear that (\cite{dp}, \cite{waa})
 \bea
  h=0\Leftrightarrow h'=0,\;\;h'^{2}=(k+1)\phi^2(\Leftrightarrow h^{2}=(k+1)\phi^2).\label{2.7}
\eea
In \cite{dp}, Dileo and Pastore introduced the notion of $(k,\mu)'$-nullity distribution, on an almost  Kenmotsu manifold  $(M, \phi, \xi, \eta, g)$, which is defined for any $p \in M$ and $k,\mu \in \mathbb {R}$ as follows:
\bea
 \nonumber N_{p}(k,\mu)'&=&\{Z\in T_{p}(M):R(X,Y)Z = k[g(Y,Z)X-g(X,Z)Y]\\&&+\mu[g(Y,Z)h'X-g(X,Z)h'Y]\}.\label{2.8}
 \eea

\noindent In \cite{dp}, it is proved that in a $(k,\mu)'$-almost Kenmotsu manifold $M^{2n+1}$ with $h' \neq 0$, $k < -1,\;\mu = -2$ and Spec$(h') = \{0,\delta,-\delta\}$, with $0$ as simple eigen value and $\delta = \sqrt{-k-1}$.

\noindent In \cite{wl}, Wang and Liu proved that for a $(2n + 1)$-dimensional $(k,\mu)'$-almost Kenmotsu manifold $M$ with $h' \neq 0$, the Ricci operator $Q$ of $M$ is given by
\bea
Q=-2nid+2n(k+1)\eta\otimes\xi-2nh'.\label{2.9}
\eea
Moreover, the scalar curvature of $M^{2n+1}$ is $2n(k-2n)$. From (\ref{2.7}), we have
\bea
R(X,Y)\xi=k[\eta(Y)X-\eta(X)Y]+\mu[\eta(Y)h'X-\eta(X)h'Y],\label{2.10}
\eea
where
$k,\mu\in\mathbb R.$ Also we get from (\ref{2.10})
\bea
R(\xi,X)Y = k[g(X,Y)\xi - \eta(Y)X] + \mu[g(h'X,Y)\xi - \eta(Y)h'X].\label{2.11}
 \eea
Contracting $X$ in (\ref{2.10}), we have
\bea
S(Y,\xi)=2nk\eta(Y). \label{2.12}
\eea
Using (\ref{2.3}), we have
\bea
(\nabla_X \eta)Y = g(X,Y) - \eta(X)\eta(Y) + g(h'X,Y). \label{2.13}
\eea
For further details on almost Kenmotsu manifolds, we refer the reader to go through the references (\cite{dey2}-\cite{dp}, \cite{pv}).

\section{\textbf{Quasi Yamabe soliton}}
In this section, we characterize $(k,\mu)'$-almost Kenmotsu manifolds admitting quasi Yamabe soliton $(g,V,\lambda,\alpha)$ such that $V$ is pointwise collinear with the characteristic vector field $\xi$. In this regard, to prove our main theorem, we need the following definition:
\begin{definition}
A vector field $V$ on an almost contact metric manifold $(M,\phi,\xi,\eta,g)$ is said to be an infinitesimal contact transformation if $\pounds_V \eta = f\eta$ for some smooth function $f$ on $M$. In particular, if $f = 0$, then $V$ is said to be a strict infinitesimal contact transformation.
\end{definition}

\begin{theorem} \label{t3.2}
If a $(2n + 1)$-dimensional $(k,\mu)'$-almost Kenmotsu manifold $M$  with $h' \neq 0$ admits quasi Yamabe soliton $(g,V,\lambda,\alpha)$ with the soliton vector field $V$ pointwise collinear with $\xi$, then
\begin{itemize}
\item[$(1)$] $V$ is a constant multiple of $\xi$.
\item[$(2)$] $V$ is a strict infinitesimal contact transformation.
\item[$(3)$] $(\pounds_V h')X = 0$ for any vector field $X$ on $M$.
\end{itemize}
\end{theorem}
\begin{proof}
If $V$ is pointwise collinear with $\xi$, then there exist a non-zero smooth function $b$ on $M$ such that $V = b\xi$. Then from (\ref{1.3}), we have
\bea
\frac{1}{2}(\pounds_{b\xi} g)(X,Y) = (r - \lambda)g(X,Y) + \alpha b^2 \eta(X)\eta(Y). \label{3.1}
\eea
Now,
\bea
\nonumber (\pounds_{b\xi} g)(X,Y) = g(\nabla_X b\xi,Y) + g(\nabla_Y b\xi,X).
\eea
Using (\ref{2.4}), the foregoing equation reduces to
\bea
\nonumber (\pounds_{b\xi} g)(X,Y) &=& (Xb)\eta(Y) + (Yb)\eta(X) \\ &&+ 2b[g(X,Y) - \eta(X)\eta(Y) - g(\phi hX,Y)]. \label{3.2}
\eea
Let $\{e_i\}$ be an orthonormal basis of the tangent space at each point of $M$. Now, substituting (\ref{3.2}) in (\ref{3.1}) and taking $X = Y = e_i$ and summing over $i$, we obtain
\bea
(\xi b) = (r - \lambda)(2n + 1) + \alpha b^2 - 2nb. \label{3.3}
\eea
Again substituting (\ref{3.2}) in (\ref{3.1}) and then putting $X = Y = \xi$ yields
\bea
(\xi b) = r - \lambda + \alpha b^2. \label{3.4}
\eea
Equating (\ref{3.3}) and (\ref{3.4}), we obtain
\bea
b = r - \lambda. \label{3.5}
\eea
Since $\lambda$ is constant and $r = 2n(k - 2n)$ is constant, $b$ is also a constant. This proves $(1)$.\\
Now, from (\ref{3.4}), we have
\bea
\alpha b^2 = - (r - \lambda), \label{3.6}
\eea
which implies $\alpha$ is also a constant. Now, since $V = b\xi$, we write (\ref{1.3}) as
 \bea
 \frac{1}{2}(\pounds_V g)(X,Y) = (r - \lambda)g(X,Y) + \alpha b^2\eta(X)\eta(Y). \label{3.7}
 \eea
Differentiating the above equation covariantly along any vector field $Z$, we get
\bea
\frac{1}{2}(\nabla_Z \pounds_V g)(X,Y) =  \alpha b^2[\eta(Y)(\nabla_Z \eta)X + \eta(X)(\nabla_Z \eta)Y]. \label{3.8}
\eea
It is well known that (see \cite{yano})
\be
\nonumber (\pounds_V \nabla_X g - \nabla_X \pounds_V g - \nabla_{[V,X]}g)(Y,Z) = -g((\pounds_V \nabla)(X,Y),Z) - g((\pounds_V \nabla)(X,Z),Y).
\ee
Since $\nabla g = 0$, then the above relation becomes
\bea
(\nabla_X \pounds_V g)(Y,Z) = g((\pounds_V \nabla)(X,Y),Z) + g((\pounds_V \nabla)(X,Z),Y). \label{3.9}
\eea
Since $\pounds_V \nabla$ is symmetric, then it follows from (\ref{3.9}) that
\bea
\nonumber g((\pounds_V \nabla)(X,Y),Z) &=& \frac{1}{2}(\nabla_X \pounds_V g)(Y,Z) + \frac{1}{2}(\nabla_Y \pounds_V g)(X,Z) \\ && - \frac{1}{2}(\nabla_Z \pounds_V g)(X,Y). \label{3.10}
\eea
Using (\ref{2.13}) and (\ref{3.8}) in (\ref{3.10}) we have
\bea
\nonumber g((\pounds_V \nabla)(X,Y),Z) =  2\alpha b^2[g(X,Y) - \eta(X)\eta(Y) + g(h'X,Y)]\eta(Z),
\eea
which implies
\bea\label{3.11}
 (\pounds_V \nabla)(X,Y) = 2\alpha b^2[g(X,Y) - \eta(X)\eta(Y) + g(h'X,Y)]\xi.
\eea
Substituting $Y = \xi$ in (\ref{3.11}), we get $(\pounds_V \nabla)(X,\xi) = 0$. From which we obtain $\nabla_Y (\pounds_V \nabla)(X,\xi) = 0$. This gives
\bea
\nonumber (\nabla_Y \pounds_V \nabla)(X,\xi) = - (\pounds_V \nabla)(\nabla_Y X,\xi) - (\pounds_V \nabla)(X,\nabla_Y \xi).
\eea
Using $(\pounds_V \nabla)(X,\xi) = 0$, (\ref{3.11}) and (\ref{2.4}) in the foregoing equation, we infer that
\bea
\nonumber (\nabla_Y \pounds_V \nabla)(X,\xi) &=& - 2\alpha b^2[ g(X,Y) - \eta(X)\eta(Y) \\ &&+ 2g(h'X,Y) + g(h'^2X,Y)]\xi. \label{3.12}
\eea
Using the foregoing equation in the following formula (see \cite{yano})
\bea
\nonumber (\pounds_V R)(X,Y)Z = (\nabla_X \pounds_V \nabla)(Y,Z) - (\nabla_Y \pounds_V \nabla)(X,Z),
\eea
we obtain
\bea
(\pounds_V R)(X,\xi)\xi = (\nabla_X \pounds_V \nabla)(\xi,\xi) - (\nabla_\xi \pounds_V \nabla)(X,\xi) = 0. \label{3.13}
\eea
Now, substituting $Y = \xi$ in (\ref{3.7}) and using (\ref{3.6}), we get $(\pounds_V g)(X,\xi)  = 0$, which implies
\bea
(\pounds_V \eta)X = g(X,\pounds_V \xi). \label{3.14}
\eea
Since $V = b\xi$ and $b$ is a constant, then we can easily obtain $\pounds_V \xi = 0$ and hence, from (\ref{3.14}), we have $(\pounds_V \eta)X = 0$ for any vector field $X$ on $M$. Therefore, $V$ is a strict infinitesimal contact transformation. This proves $(2)$.\\
From (\ref{2.10}), we have
\bea
 R(X,\xi)\xi = k(X - \eta(X)\xi) - 2h'X. \label{3.15}
\eea
Now, using (\ref{3.14})-(\ref{3.15}) and (\ref{2.10})-(\ref{2.11}), we obtain
\bea
 (\pounds_V R)(X,\xi)\xi = - 2(\pounds_V h')X. \label{3.16}
\eea
Equating (\ref{3.13}) and (\ref{3.16}), we obtain
\bea
(\pounds_V h')X = 0, \label{3.17}
\eea
for any vector field $X$ on $M$. This proves $(3)$ and the proof is complete.
\end{proof}

\begin{remark}
From (\ref{3.5}), we can see that $\lambda = r - b$. Since $r = 2n(k - 2n) < 0$ as $k < -1$, then $\lambda < 0$ when $b > 0$, $\lambda = 0$ when $b = 2n(k - 2n)$ and $\lambda > 0$ when $b < 2n(k - 2n)$. Hence, the quasi Yamabe soliton is shrinking, steady or expanding according as $b > 0$, $b = 2n(k - 2n)$ or $b < 2n(k - 2n)$ respectively.
\end{remark}
\begin{remark}
It is known that for smooth tensor field $T$, $\pounds_X T = 0$ if and only if $\phi_t$ is a symmetric transformation for $T$, where $\{\phi_t: t\in \mathbb R\}$ is the $1$-parameter group of diffeomorphisms corresponding to the vector filed $X$ on a manifold \cite{rim}. Since $h'$ is a smooth tensor field of type $(1,1)$ on $M$, then $(\pounds_V h')X = 0$ if and only if  $\psi_t$ is a symmetric transformation for $h'$, where $\{\psi_t: t\in \mathbb R\}$ is the $1$-parameter group of diffeomorphisms corresponding to the vector filed $V$.
\end{remark}

\section{\textbf{Example of a shrinking quasi Yamabe soliton}}
 In \cite{dp}, Dileo and Pastore presented an example of a $(2n + 1)$-dimensional $(k,\mu)'$-almost Kenmotsu manifold. In \cite{dey1}, the authors studied this for $5$-dimensional case and obtained $k = -2$.\\
 We now mention the necessary results from that example to verify our result.\\
$$ [\xi,e_2] = -2e_2,\; [\xi,e_3]=-2e_3,\; [\xi,e_4] = 0,\; [\xi,e_5] = 0,   $$
$ [e_i,e_j] = 0, $ where $ i, j = 2,3,4,5.$
The Riemannian metric is defined by
$$ g(\xi,\xi) = g(e_2,e_2) = g(e_3,e_3) = g(e_4,e_4) = g(e_5,e_5) = 1 $$
and $ g(\xi,e_i) = g(e_i,e_j) =0 $ for $i \neq j$; $ i, j = 2,3,4,5.$ \\
The 1-form $\eta$ is defined by $ \eta(Z) = g(Z,\xi) $, for any $Z \in T(M)$.\\
Also,  $ h'\xi = 0, \; h'e_2 = e_2, \; h'e_3 = e_3, \; h'e_4 = -e_4, \; h'e_5 = -e_5.$ \\
 The scalar curvarure is $r = -80$.\\
Now, it is easy to see that $$(\pounds_\xi g)(\xi,\xi) = (\pounds_\xi g)(e_4,e_4) = (\pounds_\xi g)(e_5,e_5) = 0,$$
$$(\pounds_\xi g)(e_2,e_2) = (\pounds_\xi g)(e_3,e_3) = 4.$$
Now, it can be easily verified from (\ref{1.3}) that $(g,V,\lambda,\alpha) = (g,\xi,-41,121)$ is a shrinking quasi Yamabe soliton, where $\lambda = -49 < 0$ and $b = 1 > 0$.\\
Also, $(\pounds_\xi \eta)X = 0$ and $(\pounds_\xi h')X = 0$ for all $X \in \{e_2, e_3, e_4, e_5\}$.\\
This verifies our theorem \ref{t3.2}.

\end{document}